\documentclass[12pt, reqno]{amsart}
\usepackage{amsmath, amsthm, amscd, amsfonts, amssymb, mathrsfs, graphicx, color}
\usepackage[colorlinks=true, urlcolor=blue, linkcolor=blue, citecolor=blue]{hyperref}

\newtheorem{theorem}{Theorem}[section]
\newtheorem{lemma}[theorem]{Lemma}

\newtheorem{corollary}[theorem]{Corollary}
\theoremstyle{definition}
\newtheorem{definition}[theorem]{Definition}

\theoremstyle{remark}
\newtheorem{remark}[theorem]{Remark}
\numberwithin{equation}{section}

\begin{document}
\title[A New Class of Operator Monotone Functions via Operator Means ] {A New Class of Operator Monotone Functions via Operator Means}
\author[R. Pal, M. Singh, M.S. Moslehian, J.S. Aujla]{Rajinder Pal$^1$ , Mandeep Singh$^1$  , Mohammad Sal Moslehian$^*$$^2$ and Jaspal Singh Aujla$^3$}
\address{$^1$ Department of Mathematics, Sant Longowal Institute of Engineering and
Technology, Longowal-148106, Punjab, India}
\email{rajinder\_singla83@yahoo.com} \email{msrawla@yahoo.com} \address{$^2$ Department of Pure Mathematics, Center of Excellence in
Analysis on Algebraic Structures (CEAAS), Ferdowsi University of
Mashhad, P. O. Box 1159, Mashhad 91775, Iran}
\email{moslehian@um.ac.ir; moslehian@member.ams.org}
\address{$^3$ Department of Mathematics, National Institute of Technology, Jalandhar-144011, Punjab, India}
\email{aujlajs@nitj.ac.in}

\subjclass[2010]{Primary 46L07, 47A63, 47A64, 47A12.}
\keywords{Operator monotone function; positive operator; positive operator; operator mean; invariance of mean.\newline \indent $^{*}$ Corresponding author}

\begin{abstract}
In this paper, we obtain a new class of functions, which is developed via the Hermite--Hadamard inequality for convex functions. The well-known one-one correspondence between the class of operator monotone functions and operator connections declares that the obtained class represents the weighted logarithmic means. We shall also consider weighted identric mean and some relationships between various operator means. Among many things, we extended the weighted arithmetic--geometric operator mean inequality as $A\#_{t}B\leq A\ell_t B\leq \frac{1}{2}(A\#_{t}B + A\nabla_{t} B)\le A\nabla_tB$ and $A\#_{t}B\leq A\mathcal{I}_t B\leq A\nabla_{t} B$ involving the considered operator means.
\end{abstract}
\maketitle

\section{Introduction and preliminaries}

In classical analysis, a function $\mathcal{M}:[0,\infty) \times [0,\infty) \to
[0,\infty)$ is called a \emph{mean} if
\begin{enumerate}
 \item $\mathcal{M}(a,b)\geq 0,$
 \item $a\leq \mathcal{M}(a,b)\leq b$ if $a\leq b,$
 \item $\mathcal{M}(a,b)=\mathcal{M}(b,a)$ (symmetry),
 \item $\mathcal{M}(a,b)$ is monotone increasing in both $a$ and $b,$
 \item $\mathcal{M}(\alpha a,\alpha b)=\alpha \mathcal{M}( a, b)$ for all $\alpha>0$, (homogeneity),
 \item $\mathcal{M}(a,b)$ is continuous in $a, \;b.$
 \end{enumerate}

Throughout this paper, the arithmetic, logarithmic, geometric, and harmonic means for positive scalars $a$ and $b$ are denoted by $A(a,b)(=\frac{a+b}{2})$, $L(a,b)(=\frac{a-b}{\log a-\log b}=\int\limits_{0}^{1}a^{t}b^{1-t}dt)$, $G(a,b)(=\sqrt{ab})$ and $H(a,b)(=(\frac{a^{-1}+b^{-1}}{2})^{-1})$, respectively. The so-called Stolarsky generalized logarithmic means $S_r(a,b)$ are defined to be $\left(\frac{a^r-b^r}{r(a-b)}\right)^{1/(r-1)}$ for $r \neq 0, 1; a, b >0, a\neq b$. It is known that $\lim_{r \to 0}S_r(a,b)=L(a,b)$, $S_2(a,b)=A(a,b)$ and $S_{-1}(a,b)=G(a,b)$ as well as $S_r$ is a continuous strictly increasing function, so
\begin{eqnarray}\label{msm1} \min\{a,b\} \leq H(a,b)\leq G(a,b)\leq L(a,b)\leq A(a,b)\leq \max\{a,b\}\,.\end{eqnarray}
Several authors studied inequality \eqref{msm1} using power means or other different techniques, see for instance \cite{FU,JS, HOA}. More generally, we consider for $t\in [0,1]$, the means $A_{t}(a,b)(=(1-t)a+tb)$, $G_{t}(a,b)(=a^{1-t}b^{t})$ and $H_{t}(a,b)(=((1-t) a^{-1}+tb^{-1})^{-1})$ as weighted arithmetic, weighted geometric and weighted harmonic means, respectively. An arbitrary weighted mean is denoted by $M_{t}(a,b)$. For such means one may trivially see that the symmetric property (3) may be lost. Again, it is easy to observe that five properties out of the above said six properties are satisfied by $M_{t}(a,b)$, while (3) is replaced by $M_{t}(a,b)=M_{1-t}(b,a)$.

 We now directly switch over to the operators and their means which are especially of great importance in several branches of sciences. \noindent
Let $\mathbb{B}(\mathscr{H})$ denote the algebra of all bounded
linear operators on a complex separable Hilbert space
$( \mathscr{H},\langle \cdot,\cdot\rangle)$ with the identity $I$. The cone of positive
operators is denoted by $\mathbb{B}(\mathscr{H})_{+}.$
For selfadjoint operators $A,B \in
\mathscr{B}(\mathscr{H})$, by $A\geq B$, we mean $A-B\in \mathbb{B}(\mathscr{H})_{+}.$ The operator means were first studied by Anderson and Duffin \cite{A-D}. The axiomatic theory for connections and operator means for positive operators acting on Hilbert space was established by Kubo and Ando \cite{KA}.
 A binary operation $\sigma: {\mathbb B}({\mathscr H})_+ \times {\mathbb B}({\mathscr H})_+ \to {\mathbb B}({\mathscr H})_+$ is called an \emph{operator mean} provided that

 (i) $A\leq C$ and $B\leq D$ imply $A\sigma B\leq C\sigma D$;

 (ii) $C^*(A\sigma B)C\leq(C^*AC)\sigma(C^*BC)$;

 (iii) $A_n\downarrow A$ and $B_n\downarrow B$ imply $(A_n\sigma B_n)\downarrow A\sigma B$, where $A_n\downarrow A$ means that $A_1\geq A_2\geq \cdots$ and $A_n\to A$ in the strong operator topology as $n\to \infty$;

 (iv) $I\sigma I=I$.\\
There exists an affine order isomorphism between the class of connections and the class of positive operator monotone functions $f$ defined on $[0,\infty)$ via $f(x)I=I\sigma (xI)\,\,(x\geq 0)$. The operator monotone function $f$ is called the \emph{representing function} of $\sigma$. Moreover, the map $\sigma \leftrightarrow f$ preserves orders in the sense that
$$A\sigma_1 B \leq A \sigma_2 B\quad (A,B \in {\mathbb B}({\mathscr H})_+) \Longleftrightarrow f_{\sigma_1}(x) \leq f_{\sigma_2}(x)\,\,\,(x\geq 0)\,.$$
The operator means corresponding to the positive operator monotone functions $(1-t)x+t$, $((1-t)x^{-1}+t)^{-1}$ and $x^{t}$, are the weighted operator arithmetic mean $A\nabla_t B=(1-t)A+tB$, the weighted operator harmonic mean $A!_tB=(((1-t)A^{-1}+tB^{-1})^{-1}$ and the weighted operator geometric mean $A\#_tB=A^{\frac{1}{2}}\left(A^{\frac{-1}{2}}BA^{\frac{-1}{2}}\right)^{t}A^{\frac{1}{2}}$, respectively. When $t=1/2$, we remove the index $t$. For more details on operator monotone functions the reader may be referred to \cite{seo, FJPS}. We shall denote as usual the operator weighted arithmetic, weighted logarithmic, weighted geometric and weighted harmonic means by $\nabla _t$, $\ell _t$, $\# _t$ and $!_t$, respectively.
Again when $t=1/2$ we remove the index $t.$

A mean $\mathcal{K}$ is called \emph{invariant} with respect to means $\mathcal{M}, \mathcal{N}$ if $\mathcal{K}(\mathcal{M}(a,b),\mathcal{N}(a,b))=\mathcal{K}(a,b)$ for all $a,b$. In the particular case where
$\mathcal{K}$ is the arithmetic mean, the equation above gives rise to the
so-called Sat\^{o}--Matkowski equation $\mathcal{M}(a,b)+\mathcal{N}(a,b)=a+b$ (see \cite{MAT1}). Another significant example is
\begin{eqnarray}\label{GAHG}
G(A(a,b),H(a,b))=G(a,b)\,.
\end{eqnarray}
In this paper, we shall first prove a new and generalized version of Hermite--Hadamard-inequality for convex integrable functions. Then, we give the notions and identify weighted versions of logarithmic mean $(L_t(a,b))$ for positive real numbers $a$ and $b$ and $t\in [0,1].$ We then prove a weighted version of inequality (1.1). This all is included in Section 2.

In Section 3, we shall concentrate on corresponding operator monotone functions and provide the possible inequalities in operator versions of the weighted means by showing the operator settings of (1.1) as
\begin{eqnarray}\label{4} A!_{t}B\leq A\#_{t}B\leq A\ell_t B \leq A\nabla_{t} B,\end{eqnarray}
for all positive invertible operators $A, B \in\mathbb{B}(\mathscr{H})_+$ and $t\in [0,1]$ being the weight. In Section 4 we shall consider weighted identric mean. Finally, in last section, we investigate the invariance of operator means, establish an operator version of equation \eqref{GAHG} and present some relationships between various operator means.

\section{ Weighted Logarithmic Mean}

The classical Hermite--Hadamard inequality provides estimates of the mean value
of a continuous convex function $f:[a, b] \rightarrow \mathbb{R},$
$$f\left(\frac{a+b}{2}\right)\leq\frac{1}{b-a} \int_a^{b}f(x)dx\leq \frac{f(a)+f(b)}{2}.$$
The history of this inequality begins with the papers of Hermite and J.
Hadamard in the years 1883--1893 (see, \cite{MOS1, Cp} and the references therein for some historical notes on the Hermite--Hadamard inequality). This inequality has triggered a huge amount of interest over the years. For instance
see \cite{M} for details on this topic.
An interesting problem related to the Hermite--Hadamard inequality is the precision
in this inequality. On making use of some fundamental techniques, we conclude the following refinement of it, which plays a key role in our results.

\begin{theorem}\label{14T} Let $f:[a,b]\rightarrow \mathbb{R}$ be a convex Riemann integrable function. Then
\begin{eqnarray}
\label{14} f(tb+(1-t)a)\hspace{-0.7cm}&&\leq (1-t)\int_0^{1}f(t\alpha(b-a)+a)d\alpha+t\int_0^{1}f((1-t)\alpha(b-a)+tb+(1-t)a) d\alpha\nonumber\\
&& \leq tf(b)+(1-t)f(a)
\end{eqnarray}
for all $t\in[0,1]$ .
\end{theorem}
\begin{proof}
We shall first prove the last inequality in \eqref{14}, i.e.,
\begin{eqnarray}\label{15}(1-t)\int_0^{1}f(t\alpha(b-a)+a)d\alpha\hspace{-0.6cm}&&+t\int_0^{1}f((1-t)\alpha(b-a)+tb+(1-t)a) d\alpha\nonumber\\
&&\leq tf(b)+(1-t)f(a).\end{eqnarray}
In fact, the left side of \eqref{15} can be written as
\begin{align}\label{16}
&(1-t)\int_0^{1}f(t\alpha b+(1-t\alpha)a)d\alpha+t\int_0^{1}f((1-t)(\alpha(b-a)+a)+tb) d\alpha\nonumber\\
&\leq(1-t)\int_0^{1}(t\alpha f(b)+(1-t\alpha)f(a))d\alpha+t\int_0^{1}((1-t)f((\alpha(b-a)+a))+tf(b)) d\alpha\nonumber\\
&\leq(1-t)\int_0^{1}(t\alpha f(b)+(1-t\alpha)f(a))d\alpha+t\int_0^{1}((1-t)(\alpha f(b)+(1-\alpha)f(a))+tf(b)) d\alpha\nonumber\\
&=tf(b)+(1-t)f(a).
\nonumber\end{align}
For the proof of first inequality in \eqref{14}, we write
\begin{align*}
&(1-t)\int_0^{1}f(t\alpha(b-a)+a)d\alpha+t\int_0^{1}f((1-t)\alpha(b-a)+tb+(1-t)a) d\alpha\\
&=(1-t)\int_0^{1}f(t(\alpha b+(1-\alpha)a)+(1-t)a)d\alpha+t\int_0^{1}f((1-t)(\alpha b+(1-\alpha)a)+tb) d\alpha\\
&=(1-t)\int_0^{1}f(t(\alpha b+(1-\alpha)a)+(1-t)a)d\alpha+t\int_0^{1}f((1-t)((1-\alpha) b+\alpha a)+tb) d\alpha\\
&\geq\int_0^{1}f((1-t)t(\alpha b+(1-\alpha)a)+(1-t)^{2}a+t(1-t)((1-\alpha) b+\alpha a)+t^2 b) d\alpha\nonumber\\
&=f(tb+(1-t)a).
\end{align*}
\end{proof}

\begin{theorem} \label{t2.2} For $a,b\in\mathbb{R}^+$ and $t\in (0,1)$ it holds that
\begin{eqnarray}\label{6a}
a^{1-t}b^t\leq\frac{1}{\log a-\log b}\left(\frac{1-t}{t}a^{1-t}(a^t-b^t)+\frac{t}{1-t}b^t(a^{1-t}- b^{1-t})\right)
\leq (1-t)a+tb.
\end{eqnarray}
\end{theorem}
\begin{proof} On taking $f(x)=e^x$ in Theorem \ref{14T}, we obtain
\begin{eqnarray}
\label{20} e^{tb+(1-t)a}\hspace{-0.7cm}&&\leq (1-t)\int_0^{1}e^{t(b-a)\alpha+a}d\alpha+t\int_0^{1}e^{(1-t)(b-a)\alpha+tb+(1-t)a} d\alpha\nonumber\\
&& \leq te^b+(1-t)e^a.
\end{eqnarray}
Calculating the integrals in \eqref{20}, we get
\begin{eqnarray}
\label{21} \hspace{-3cm} (1-t)\int_0^{1}e^{t(b-a)\alpha+a}d\alpha+t\hspace{-0.7cm}&&\int_0^{1}e^{(1-t)(b-a)\alpha+tb+(1-t)a} d\alpha\nonumber\\
&& =\frac{1}{b-a}\left(\frac{1-t}{t}(e^{tb+(1-t)a}-e^a)+\frac{t}{1-t}(e^b-e^{tb+(1-t)a})\right).
\end{eqnarray}
Using \eqref{21} in \eqref{20} and replacing $e^a$, $e^b$ by $a$, $b$, respectively, we get the required inequality \eqref{6a}.
\end{proof}

Next we introduce the weighted logarithmic mean $L_t(a,b)$ of two positive numbers $a,b$ for $t\in (0,1)$, as
\begin{eqnarray}
\label{6b}L_t(a,b)=\frac{1}{\log a-\log b}\left(\frac{1-t}{t}a^{1-t}(a^t-b^t)+\frac{t}{1-t}b^t(a^{1-t}- b^{1-t})\right).\nonumber\end{eqnarray}
$L_0(a,b)$ and $L_1(a,b)$ are defined to be $L_0(a,b)= \lim\limits_{t\rightarrow 0}^{}L_t(a,b)$ and $L_1(a,b)=\lim\limits_{t\rightarrow 1}^{}L_t(a,b)$
We observe that $\lim\limits_{t\rightarrow 0}^{}L_t(a,b)= a$, $\lim\limits_{t\rightarrow 1}^{}L_t(a,b)=b$ and $L_{1/2}(a,b)=
L(a,b).$

Moreover, we easily see that $L_t(a,b)$ satisfies all the properties given above for a weighted mean.

To prove (1), (2), (4), (5) and (6) one needs the following equivalent expressions of $L_t(a,b)$
\begin{eqnarray}&&\hspace{-1.50cm}
\label{6c}\frac{1}{\log a-\log b}\left(\frac{1-t}{t}a^{1-t}(a^t-b^t)+\frac{t}{1-t}b^t(a^{1-t}- b^{1-t})\right)\nonumber\\
&&\hspace{1.9cm}=(1-t) a^{1-t}\int_0^1a^{tx}b^{t(1-x)}dx +t b^t\int_0^1a^{(1-t)x}b^{(1-t)(1-x)}dx\nonumber\\
&&\hspace{1.9cm}= (1-t) \int_0^1a^{1-t(1-x)}b^{t(1-x)}dx +t \int_0^1a^{(1-t)x}b^{1-x(1-t)}dx.
\end{eqnarray}
By putting $t(1-x) = z$ and $(1-t)x=w$ in the first and the second part of \eqref{6c}, respectively, we get
\begin{eqnarray}&&\hspace{-1.50cm}\label{7} L_t(a,b)=\frac{1-t}{t} \int_0^t a^{1-z}b^zdz +\frac{t}{1-t} \int_0^{1-t}a^{w} b^{1-w}dw\nonumber\\
&&=\frac{1-t}{t} \int_0^t a^{1-z}b^zdz +\frac{t}{1-t} \int_t^1a^{1-z}b^z dz .
\end{eqnarray}
Note also that $a\leq a^{1-z}b^z \leq b$ whenever $a\leq b$, and $z\in [0,1]$. Further, a change of variables technique for the integration and use of \eqref{7} prove $L_t(a,b)= L_{1-t}(b,a).$ Further, it follows from Theorem \ref{t2.2} that
$$
H_t(a,b)\le G_t(a,b)\le L_t(a,b)\le A_t(a,b).
$$
We remark here that the authors in \cite{FU} proved a comparison of the logarithmic mean and the Heronian mean $(=\frac{2}{3}\sqrt{ab}+\frac{1}{3}\left(\frac{a+b}{2}\right))$, in the following way.
\begin{eqnarray}\label{fu}L(a,b)=\frac{a-b}{\log a-\log b}\leq \frac{2}{3}\sqrt{ab}+\frac{1}{3}\left(\frac{a+b}{2}\right).\end{eqnarray}
However, we claim that \eqref{fu} is not true in general, i.e. for weighted version.
We prove this by furnishing the following example:\\
On taking $a=706,\; b=31.8,$ and $t=0.2169,$ we obtain
$$L_t(a,b)=\frac{1}{\log a-\log b}\left(\frac{1-t}{t}a^{1-t}(a^t-b^t)+\frac{t}{1-t}b^t(a^{1-t}- b^{1-t})\right) = 431.8506$$
while, $$\frac{2}{3}a^{1-t}b^t+\frac{1}{3}((1-t)a+tb)=426.8502.$$

Next, we establish the class of operator monotone functions corresponding to the weighted logarithmic mean for operators. We shall also prove the classical inequalities \eqref{4}.

We easily compute the representing function $f_{t}(x)$ for weighted logarithmic mean using \eqref{6c} as,
\begin{eqnarray}\label{6}
f_{t}(x)=\frac{1}{\log x}\left(\frac{1-t}{t}(x^t-1)+\frac{t}{1-t}x^t( x^{1-t}-1)\right),\nonumber\end{eqnarray}
$t\in (0,1)$, $f_0(x) = \lim\limits_{t\rightarrow 0}^{}f_t(x),$ $f_1(x)=\lim\limits_{t\rightarrow 1}^{}f_t(x)$ or equivalently
\begin{eqnarray}\label{8}
f_{t}(x)=(1-t) \int_0^{1}x^{t\alpha}d\alpha+tx^{t}\int_0^{1}x^{(1-t)\alpha}d\alpha.
\end{eqnarray}
The change of variable technique entails further the following two more equivalent forms of \eqref{6}, which will be used in the sequel.
\begin{eqnarray}
f_{t}(x)&=&(1-t) \int_0^{1}x^{t(1-\alpha)}d\alpha+tx^{t}\int_0^{1}x^{(1-t)\alpha}d\alpha\label{9}\\
&=& \label{10}(1-t) \int_0^{1}x^{t\alpha}d\alpha+t\int_0^{1}x^{1-\alpha+t\alpha}d\alpha \label{999}\\
&=& \frac{1-t}{t}\int_0^tx^\alpha d\alpha+\frac{t}{1-t}\int_t^1x^\alpha d\alpha\,\nonumber.
\end{eqnarray}
This is clear from \eqref{999} that the class of functions $f_{t}(x)$ for $t \in [0,1]$ are operator monotone.\\

 We now prove the following lemma which
will provide a tool in proving inequality \eqref{4}.
\begin{lemma}\label{le} For $x\geq 0$ and $t\in[0,1]$, the inequality
\begin{eqnarray}
\label{1111} x^t \leq f_{t}(x)\leq \frac{1}{2}\left(x^t+1-t+tx\right).
\end{eqnarray}
is valid.
\end{lemma}
\begin{proof}
To prove the first inequality in \eqref{1111}, we use \eqref{999} for $f_{t}(x)$ to obtain
\begin{eqnarray*}
\label{12} (1-t) \int_0^{1}x^{t\alpha}d\alpha+t\int_0^{1}x^{1-(1-t)\alpha}d\alpha \hspace{-0.6cm}&&\geq \int_0^{1}x^{(1-t)t\alpha} x^{t(1-t)(1-\alpha)+t^2}d\alpha  =\int_0^{1}x^{(1-t)t+t^2}d\alpha\nonumber = x^t.
\end{eqnarray*}
To prove the latter inequality in \eqref{1111}, we use \eqref{8} for $f_{t}(x)$ to get
\begin{eqnarray}
(1-t) \int_0^{1}x^{t(1-\alpha)}d\alpha+tx^{t}\int_0^{1}x^{(1-t)\alpha}d\alpha &&\nonumber\\
 &&\hspace{-1.6cm}\leq \int_0^{1}(1-t)((1-\alpha)x^{t}+\alpha)+tx^{t}(\alpha x^{1-t}+1-\alpha)d\alpha\nonumber\\
 &&\hspace{-1.6cm}=\int_0^{1}\big((1-\alpha)x^{t}+(1-t)\alpha+\alpha tx\big)d\alpha\nonumber \\
 &&\hspace{-1.6cm}=\frac{1}{2}\left(x^t+1-t+tx\right).\nonumber
\end{eqnarray}
\end{proof}
\begin{theorem}\label{the}
Let $A,B$ be invertible operators in $\mathbb{B}(\mathscr{H})_{+}$ and $t\in [0,1].$ Then
\begin{eqnarray}\label{113} A!_{t}B\leq A\#_{t}B\leq A\ell_t B\leq \frac{1}{2}(A\#_{t}B + A\nabla_{t} B)\le A\nabla_tB.\end{eqnarray}
\end{theorem}
\begin{proof}It is sufficient to prove the middle double inequality, since $A!_{t}B\leq A\#_{t}B$ and $A\#_tB\le A\nabla_tB$ are the well known weighted harmonic-geometric and weighted geometric-arithmetic mean inequalities. Replacing $x$ by $A^{-1/2}BA^{-1/2}$ in Lemma \ref{le}, we obtain
\begin{eqnarray}\label{26} I\#_{t}A^{-1/2}BA^{-1/2}\leq I\ell_t A^{-1/2}BA^{-1/2}\leq \frac{1}{2}(I\#_{t}A^{-1/2}BA^{-1/2} + I\nabla_{t} A^{-1/2}BA^{-1/2}).\end{eqnarray}
Now, pre and post multiplying \eqref{26} by $A^{1/2}$, we get the required result.
\end{proof}
\begin{remark}In view of $A\#_{t}B \leq A\nabla_{t} B$, inequality \eqref{113} is better than last inequality in \eqref{4}. Moreover, inequality \eqref{113} is establishing a comparison of weighted logarithmic and weighted Heronian means for operators. For some operator inequalities regarding Heronian mean see \cite{KK}.
\end{remark}

\section{Weighted Identric Mean}

The identric mean for $a,b\in \mathbb{R}^+ $ is defined to be $I(a,b)=\frac{1}{e}\left(\frac{b^b}{a^a}\right)^\frac{1}{b-a}.$ In this section we introduce weighted identric mean and identify its representing function. We also prove an operator inequality involving identric mean in this section.

\begin{theorem}\label{29} For $a,b\in\mathbb{R}^+$ and $t\in[0,1]$ it holds that
\begin{eqnarray}
\label{22}a^{1-t}b^t\leq \frac{1}{e}((1-t)a+tb)^{\frac{(1-2t)(tb+(1-t)a)}{t(1-t)(b-a)}}\left(\frac{b^\frac{tb}{1-t}}{a^\frac{(1-t)a}{t}}\right)^\frac{1}{b-a}
\leq (1-t)a+tb.\end{eqnarray}
\end{theorem}
\begin{proof} On taking $f(x)=-\log x$ in Theorem \ref{14T}, we obtain
\begin{eqnarray}
\label{23} \log(tb+\hspace{-0.7cm}&&(1-t)a)\nonumber\\
\hspace{-0.7cm}&&\geq (1-t)\int_0^{1}\log(t(b-a)\alpha+a)d\alpha+t\int_0^{1}\log((1-t)(b-a)\alpha+tb+(1-t)a) d\alpha\nonumber\\
&& \geq t\log b+(1-t)\log a.
\end{eqnarray}
Now, the calculations of the middle part of the above inequality are given by
\begin{eqnarray}
\label{24} (1-t)\int_0^{1}\log(\hspace{-0.7cm}&&t(b-a)\alpha+a)d\alpha+t\int_0^{1}\log((1-t)(b-a)\alpha+tb+(1-t)a) d\alpha\nonumber\\
&&=\frac{1-t}{t(b-a)}\left(((1-t)a+tb)\log ((1-t)a+tb)-a\log a-t(b-a)\right)\nonumber\\
&&+\frac{t}{(1-t)(b-a)}\left(b\log b-((1-t)a+tb)\log ((1-t)a+tb)-(1-t)(b-a)\right)\nonumber\\
&&=\log\left(\frac{1}{e}((1-t)a+tb)^{\frac{(1-2t)(tb+(1-t)a)}{t(1-t)(b-a)}}\left(\frac{b^\frac{tb}{1-t}}{a^\frac{(1-t)a}{t}}\right)^\frac{1}{b-a}\right).
\end{eqnarray}
Finally, using \eqref{24} in \eqref{23} and monotonicity of exponential function, we get the required inequality \eqref{22}.
\end{proof}
We now introduce the
weighted identric mean $I_t(a,b)$ for $t\in(0,1)$ of two positive numbers $a,b$ as
\begin{eqnarray*}
\label{6b}I_t(a,b)&=&\frac{1}{e}((1-t)a+tb)^{\frac{(1-2t)(tb+(1-t)a)}
{t(1-t)(b-a)}}\left(\frac{b^\frac{tb}{1-t}}{a^\frac{(1-t)a}{t}}\right)^\frac{1}{b-a}\,.\nonumber\\
\end{eqnarray*}
$I_0(a,b)$ and $I_1(a,b)$ are defined to be $ \lim\limits_{t\rightarrow 0}^{}I_t(a,b)$ and $ \lim\limits_{t\rightarrow 1}^{}I_t(a,b)$ respectively.
We easily see that $\lim\limits_{t\rightarrow 0}^{}I_t(a,b)=a$, $\lim\limits_{t\rightarrow 1}^{}I_t(a,b)=b$, and $I_{1/2}(a,b)= I(a,b).$

Here also, we observe that $I_t(a,b)$ satisfies all the properties for any weighted mean except (4) which becomes clear from the integral equation \eqref{24}.

We shall denote by $\mathcal{I}_t$ the operator identric mean.
To see an operator inequality involving identric mean, we use the representing function $g_{t}(x)$ for this in following way,
\begin{eqnarray*}
g_{t}(x)=\frac{1}{e}(1-t+tx)^{\frac{(1-2t)(1-t+tx)}
{t(1-t)(x-1)}}\left(x^{\frac{tx}{1-t}}\right)^\frac{1}{x-1}
\end{eqnarray*}
Employing \eqref{24} we get
\begin{eqnarray}\label{msm2}
\log(g_{t})(x)= \frac{1-t}{t}\int_0^t\log(\alpha x+(1-\alpha))d\alpha +\frac{t}{1-t}\int_t^1\log(\alpha x+(1-\alpha))d\alpha
\end{eqnarray}

\begin{theorem}
Let $A,B$ be invertible operators in $\mathbb{B}(\mathscr{H})_{+}$ and $t\in [0,1].$ Then
\begin{eqnarray}\label{30} A!_{t}B\leq A\#_{t}B\leq A\mathcal{I}_t B\leq A\nabla_{t} B.\end{eqnarray}
\end{theorem}

\begin{proof} Again, as in Theorem \ref{the}, we skip to prove first inequality.
We only prove the remaining part in \eqref{30}. Using \eqref{22} with $b/a$ replaced by $x$, we obtain
\begin{eqnarray}
\label{27} x^t \leq g_{t}(x)\leq 1-t+tx.
\end{eqnarray}
Replacing $x$ by $A^{-1/2}BA^{-1/2}$ in \eqref{27}, we obtain
\begin{eqnarray}\label{31}
I\#_{t}A^{-1/2}BA^{-1/2}\leq I\mathcal{I}_t A^{-1/2}BA^{-1/2}\leq I\nabla_{t} A^{-1/2}BA^{-1/2}.
\end{eqnarray}
Now, pre and post multiplication by $A^{1/2}$ in \eqref{31} conclude the result.
\end{proof}

Finally, we would like to remark that the question of proving that the representing function $g_t(x)$ is operator monotone has eluded us. However it follows from \eqref{msm2} that the function $\log g_t(x)$ is operator monotone. It is further remarked that it is well known that the functions $g_0(x), g_{1/2}(x)$ and $g_1(x)$ are operator monotone.

\section{Invariance of operator means}

In this section, we investigate the invariance of operator means and some operator inequalities between operator means.

\begin{definition}
An operator mean $\sigma$ is called \emph{invariant} with respect to operator means $\tau, \rho$ if
\begin{eqnarray}\label{inv1}
A\sigma B=(A\tau B) \sigma (A\rho B) \qquad (A, B \in {\mathbb B}({\mathscr H})_+)\,.
\end{eqnarray}
\end{definition}

Employing the properties of operator means we observe that \eqref{inv1} holds if and only if

\begin{eqnarray*}
I\sigma A^{\frac{-1}{2}}BA^{\frac{-1}{2}}=(I\tau A^{\frac{-1}{2}}BA^{\frac{-1}{2}}) \sigma (I\rho A^{\frac{-1}{2}}BA^{\frac{-1}{2}})\qquad (A, B \in {\mathbb B}({\mathscr H})_+)\, ,
\end{eqnarray*}
which is in turn equivalent to
\begin{eqnarray}\label{inv2}
I\sigma C=(I\tau C) \sigma (I\rho C)\qquad (C \in {\mathbb B}({\mathscr H})_+)\,.
\end{eqnarray}

By the definition of the representing function \eqref{inv2} holds if and only if
\begin{eqnarray}\label{inv3}
f(C)&=&g(C)\sigma h(C)\nonumber\\
&=&g(C)^{\frac{1}{2}}f\left(g(C)^{\frac{-1}{2}}h(C)g(C)^{\frac{-1}{2}}\right)g(C)^{\frac{1}{2}}\nonumber \\
&=& g(C)f\left(g(C)^{-1}h(C)\right)\qquad (C \in {\mathbb B}({\mathscr H})_+)\,,
\end{eqnarray}
where $f, g, h$ are representing functions corresponding to $\sigma, \tau, \rho$, respectively. By the functional calculus \eqref{inv3} is true if and only if $f(t)=g(t)f\left(g(t)^{-1}h(t)\right)$ for all $t\geq 0$.

\noindent We have just proved the following theorem.
\begin{theorem}\label{main2}
Let $\sigma, \tau, \rho$ be three operator means with representing functions $f, g, h$. Then $\sigma$ is invariant with respect to operator means $\tau, \rho$ if and only if
\begin{eqnarray*}
f(t)=g(t)f\left(g(t)^{-1}h(t)\right)
\end{eqnarray*}
for all $t\geq 0$.
\end{theorem}

\begin{corollary}
The operator weighted geometric mean $\#_p$ is invariant with respect to the operator weighted geometric means $\#_q, \#_r$ if and only if $p(1-r)=q(1-p)$.
\end{corollary}
\begin{proof}
$\#_p$ is invariant with respect to $\#_q, \#_r$ if and only if $t^p=t^q(t^{-q}t^r)^p$, and this if and only if $p(1-r)=q(1-p)$.
\end{proof}

The next results concern with positivity of some block matrices and its application for finding some interrelationship between some operator means considered in this paper, see \cite{MS}.

\begin{lemma} \label{lem0}\cite[Theorem 1.3.3]{BHA}
Let $X, Z$ be (strictly) positive operators and $Y$ be an arbitrary operator. Then the block matrix
$\left[\begin{array}{cc}X&Y\\Y^*&Z\end{array}\right]$ is positive if and only if $X \geq YZ^{-1}Y^*$.
\end{lemma}

\begin{theorem}\label{main}
Let $A, B$ be strictly positive operators and let $\sigma, \tau, \rho$ be operator means with representing functions $f, g, h$, respectively. Then
\begin{eqnarray}\label{cata}
\left[\begin{array}{cc}A \sigma B& A \tau B\\A \tau B& A\rho B\end{array}\right] \geq 0\,.
\end{eqnarray}
if and only if
\begin{eqnarray}\label{fgh}
f(t)h(t) \geq g(t)^2
\end{eqnarray}
for all $t\geq 0$.
\end{theorem}
\begin{proof}
Using the functional calculus for the strictly positive operator
$A^{\frac{-1}{2}}BA^{\frac{-1}{2}}$ we observe that \eqref{fgh} is valid if and only if
\begin{eqnarray}\label{1}
f\left(A^{\frac{-1}{2}}BA^{\frac{-1}{2}}\right) \geq g\left(A^{\frac{-1}{2}}BA^{\frac{-1}{2}}\right) h\left(A^{\frac{-1}{2}}BA^{\frac{-1}{2}}\right)^{-1}g\left(A^{\frac{-1}{2}}BA^{\frac{-1}{2}}\right) \,.
\end{eqnarray}
Applying Lemma \ref{lem0}, we see that \eqref{1} holds if and only if
\begin{eqnarray}\label{11}
\left[\begin{array}{cc}f\left(A^{\frac{-1}{2}}BA^{\frac{-1}{2}}\right) & g\left(A^{\frac{-1}{2}}BA^{\frac{-1}{2}}\right)\\g\left(A^{\frac{-1}{2}}BA^{\frac{-1}{2}}\right)& h\left(A^{\frac{-1}{2}}BA^{\frac{-1}{2}}\right)\end{array}\right] \geq 0\,.
\end{eqnarray}
Clearly \eqref{11} is equivalent to
\begin{eqnarray*}
\left[\begin{array}{cc}A^{\frac{1}{2}}&0\\0&A^{\frac{1}{2}}\end{array}\right]
\left[\begin{array}{cc}f\left(A^{\frac{-1}{2}}BA^{\frac{-1}{2}}\right) & g\left(A^{\frac{-1}{2}}BA^{\frac{-1}{2}}\right)\\g\left(A^{\frac{-1}{2}}BA^{\frac{-1}{2}}\right)& h\left(A^{\frac{-1}{2}}BA^{\frac{-1}{2}}\right)\end{array}\right]
\left[\begin{array}{cc}A^{\frac{1}{2}}&0\\0&A^{\frac{1}{2}}\end{array}\right] \geq 0\,,
\end{eqnarray*}
or
\begin{eqnarray*}
\left[\begin{array}{cc}A^{\frac{1}{2}}f\left(A^{\frac{-1}{2}}BA^{\frac{-1}{2}}\right)A^{\frac{1}{2}} & A^{\frac{1}{2}}g\left(A^{\frac{-1}{2}}BA^{\frac{-1}{2}}\right)A^{\frac{1}{2}}\\
A^{\frac{1}{2}}g\left(A^{\frac{-1}{2}}BA^{\frac{-1}{2}}\right)A^{\frac{1}{2}}&
A^{\frac{1}{2}}h\left(A^{\frac{-1}{2}}BA^{\frac{-1}{2}}\right)A^{\frac{1}{2}}
\end{array}\right] \geq 0\,,
\end{eqnarray*}
which is \eqref{cata}.
\end{proof}

\begin{corollary} Let $A, B$ be strictly positive operators, Then
$$A \sharp B \leq A \ell B \# A \mathcal{I} B\,.$$
\end{corollary}
\begin{proof}
It is easy to see that $G(1,x)^2 \leq L(1,x)I(1,x)$. Considering the appropriate representation functions in Theorem \ref{main}, we reach the desired inequality.
\end{proof}


\begin{thebibliography}{99}


\bibitem{A-D} W.N. Anderson Jr. and R.J. Duffin, \textit{Series and parallel addition of matrices}, J. Math. Anal. Appl. \textbf{26} (1969), 576--594.

\bibitem{BHA} R. Bhatia, \textit{Positive Definite Matrices}, Princeton Series in Applied Mathematics. Princeton University Press, Princeton, NJ, 2007.

\bibitem{FJPS} M. Fujii, J. Micic Hot, J. Pe\v{c}ari\'c, Y. Seo, \textit{Recent developments of Mond-Pecaric method in operator inequalities. Inequalities for bounded selfadjoint operators on a Hilbert space. II}, Monographs in Inequalities, 4. ELEMENT, Zagreb, 2012.

\bibitem{HOA} D.T. Hoa, T.M. Ho and H. Osaka, \textit{Interpolation classes and matrix means}, Banach J. Math. Anal. 9 (2015), no. 3, 140--152.

\bibitem{KK} F. Kittaneh and M. Krni\'c, \textit{Refined Heinz operator inequalities}, Linear Multilinear Algebra \textbf{61} (2013), no. 8, 1148--1157.

\bibitem{KA} F. Kubo and T. Ando, \textit{Means of positive linear operators}, Math. Ann. \textbf{246} (1980), 205--224.

 \bibitem{MS} J.S. Matharu and J.S. Aujla, \textit{Some inequalities for operator means and Hadamard product}, Math. Inequal. Appl. \textbf{13} (2010), no. 3, 643--653.

\bibitem{MAT1} J. Matkowski, \textit{Invariant and complementary quasi-arithmetic means}, Aequationes Math. \textbf{57} (1) (1999) 87-107.

\bibitem{M} M. Mihailescu and C.P. Niculescu, \textit{An extension of the Hermite--Hadamard inequality through
subharmonic functions}, Glasg. Math. J. \textbf{49} (2007), 1--6.

\bibitem{MOS1} M.S. Moslehian, \textit{Matrix Hermite-Hadamard type inequalities}, Houston J. Math. \textbf{39} (2013), no. 1, 177--189.

\bibitem{Cp} C.P. Niculescu and L.-E. Persson, \textit{Convex Functions and their Applications. A Contemporary
Approach}, CMS Books in Mathematics 23, Springer-Verlag, New York, 2006.

\bibitem{seo} J. Pe\v cari\'c, T. Furuta, J. Mi\'ci\'c Hot and Y. Seo, \textit{Mond--Pe\v cari\'c method in operator inequalities. Inequalities for bounded selfadjoint operators on a Hilbert space}, Monographs in Inequalities, 1. Element, Zagreb, 2005.

\bibitem{JS}J. S\'andor, \textit{On the identric and logarithmic means}, Aequationes Math. \textbf{40} (1990), 261--270.

\bibitem{FU} F. Shigeru and K. Yanagi \textit{Bounds of the logarithmic mean}, J. Inequal. Appl. 2013, 2013:535, 11 pp.

\end{thebibliography}
\end{document}